\documentclass[12pt]{article}

\usepackage{verbatim}

\usepackage{amsfonts}

\usepackage{amsthm}

\usepackage{amssymb}

\usepackage{amsmath}

\usepackage{mathabx}

\usepackage{enumerate}

\usepackage[all]{xy}

\usepackage{graphicx}

\usepackage{hyperref}

\newcommand{\N}{\mathbb{N}}

\newcommand{\Z}{\mathbb{Z}}

\newcommand{\R}{\mathbb{R}}

\newcommand{\C}{\mathbb{C}}

\newcommand{\Hawaii}{Hawai\kern.05em`\kern.05em\relax i}

\theoremstyle{plain}
\newtheorem{theorem}{Theorem}[section]
\newtheorem{lemma}[theorem]{Lemma}
\newtheorem{corollary}[theorem]{Corollary}
\newtheorem{proposition}[theorem]{Proposition}

\newtheorem{definition-theorem}[theorem]{Definition / Theorem}

\newtheorem*{conjecture*}{Conjecture}
\newtheorem*{theorem*}{Theorem}

\theoremstyle{definition}
\newtheorem{definition}[theorem]{Definition}

\theoremstyle{remark}

\newtheorem*{example*}{Example}  
\newtheorem*{remark*}{Remark}

\begin{document}
\title{Bott periodicity and almost commuting matrices}
\author{Rufus Willett}

\maketitle

\begin{abstract}
We give a proof of the Bott periodicity theorem for topological $K$-theory of $C^*$-algebras based on Loring's treatment of Voiculescu's almost commuting matrices and Atiyah's rotation trick.  We also explain how this relates to the Dirac operator on the circle; this uses Yu's localization algebra and an associated explicit formula for the pairing between the first $K$-homology and first $K$-theory groups of a (separable) $C^*$-algebra.
\end{abstract}

\section{Introduction}

The aim of this note is to give a proof of Bott periodicity using Voiculescu's famous example \cite{Voiculescu:1983km} of `almost commuting' unitary matrices that cannot be approximated by exactly commuting unitary matrices.  Indeed, in his thesis \cite[Chapter I]{Loring:1985ud} Loring explained how the properties of Voiculescu's example can be seen as arising from Bott periodicity; this note explains how one can go `backwards' and use Loring's computations combined with Atiyah's rotation trick \cite[Section 1]{Atiyah:1968ek} to prove Bott periodicity.  Thus the existence of matrices with the properties in Voiculescu's example is in some sense equivalent to Bott periodicity.  

A secondary aim is to explain how to interpret the above in terms of Yu's localization algebra \cite{Yu:1997kb} and the Dirac operator on the circle.  The localization algebra of a topological space $X$ is a $C^*$-algebra $C_L^*(X)$ with the property that the $K$-theory of $C^*_L(X)$ is canonically isomorphic to the $K$-homology of $X$.  We explain how Voiculescu's matrices, and the computations we need to do with them, arise naturally from the Dirac and Bott operators on the circle using the localisation algebra.  A key ingredient is a new explicit formula for the pairing between $K$-homology and $K$-theory in terms of Loring's machinery.

We do not claim any real technical originality in this note: as will be obvious, our approach to Bott periodicity is heavily dependent on Loring's work in particular.  However, we think this approach is particularly attractive and concrete -- it essentially reduces the proof of the Bott periodicity theorem to some formal observations and a finite dimensional matrix computation -- and hope that this exposition is worthwhile from that point of view.  We also hope it is interesting insofar as it bridges a gap between different parts of the subject.   A secondary motivation comes from possible applications in the setting of controlled $K$-theory; this will be explored elsewhere, however.\\

\textbf{Acknowledgment}\\

Thank you to the anonymous referee for several helpful comments.  This work was partly supported by NSF grant DMS 1564281.

\section{Preliminaries and Atiyah's rotation trick}

In this section, we set up notation and recall Atiyah's rotation trick.

Let $S^1:=\{z\in \C\mid |z|=1\}$.  Let $S$ denote the $C^*$-algebra $C_0(S^1\setminus \{1\})$, so $S\cong C_0(\R)$.  For a $C^*$-algebra $A$, let $SA$ denote the suspension $S\otimes A\cong C_0(S^1\setminus \{1\},A)$.  Let $b\in K_1(S)$ denote the Bott generator, i.e.\ $b$ is the class of the unitary $z\mapsto z^{-1}$ in the unitization of $S$.  Recall moreover that for a unital $C^*$-algebra $A$, the \emph{Bott map} is defined on the class of a projection $p\in M_n(A)$ by
$$
\beta_A:K_0(A)\to K_1(SA), \quad [p] \mapsto [bp+(1-p)],
$$
where the right hand side is understood to be the class in $K_1(SA)$ of the unitary $z\mapsto z^{-1} p+1-p$ in the unitization of $C(S^1\setminus \{1\},M_n(A))\cong M_n(SA)$.  This identifies with an element in the unitary group of the $n\times n$ matrices over the unitization of $SA$.  One checks directly that the process above $\beta_A$ indeed gives a well-defined map on $K$-theory that induces a map in the non-unital case in the usual way, and moreover that the collection of maps is natural in $A$ in the sense that for any $*$-homomorphism $\phi:A\to B$, the corresponding diagram
$$
\xymatrix{ K_0(A) \ar[d]^-{\phi_*} \ar[r]^-{\beta_A} & K_1(SA) \ar[d]^-{(\text{id}_S\otimes \phi)_*}\\ K_0(B) \ar[r]^-{\beta_B} & K_1(SB) }
$$
commutes (see for example \cite[Section 11.1.1]{Rordam:2000mz}).

\begin{theorem}[Bott periodicity theorem]
For every $C^*$-algebra $A$, $\beta_A$ is an isomorphism.  
\end{theorem}

An important ingredient in our approach to this will be \emph{Atiyah's rotation trick} \cite[Section 1]{Atiyah:1968ek}: this allows one to reduce the proof of Bott periodicity to constructing a homomorphism $\alpha_A:K_1(SA)\to K_0(A)$ for each $C^*$-algebra $A$, so that the collection $\{\alpha_A\mid A\text{ a $C^*$-algebra}\}$ satisfies two natural axioms.  As well as the fact that it simplifies the proof of Bott periodicity, a key virtue the rotation trick (as already pointed out in \cite[Section 7]{Atiyah:1968ek}) is that the axioms satisfied by the $\alpha_A$ are easily checked for several different natural analytic and geometric constructions of a collection of homomorphisms $\alpha_A$; this paper essentially codifies the observation that Voiculescu's almost commuting matrices give yet another way of constructing an appropriate family $\{\alpha_A\}$.  

In order to give a precise statement of the axioms, recall from \cite[Section 4.7]{Higson:2000bs} that for unital $C^*$-algebras $A$ and $B$, the formulas 
$$
\times:K_0(A)\otimes K_0(B)\to K_0(A\otimes B),\quad [p]\otimes [q] \mapsto [p\otimes q]
$$
and 
$$
\times:K_1(A)\otimes K_0(B)\to K_1(A\otimes B), \quad [u]\otimes [p]\mapsto [u\otimes p+1\otimes (1-p)]
$$
induce canonical external products on $K$-theory; moreover, applying these products to unitizations and restricting, these extend to well-defined products in the non-unital case too.  

Here then is a precise statement of the reduction that Atiyah's rotation trick allows: see the exposition in \cite[Section 4.9]{Higson:2000bs} for a proof of the result as stated here.

\begin{proposition}\label{atiyah rot}
Assume that for each $C^*$-algebra $A$ there is homomorphism $\alpha_A:K_1(SA)\to K_0(A)$ satisfying the following conditions.
\begin{enumerate}[(i)]
\item $\alpha_\C(b)=1$.
\item For any $C^*$-algebras $A$ and $B$, the diagram below
$$
\xymatrix{ K_1(SA)\otimes K_0(B) \ar[d]^-{\alpha_A\otimes 1} \ar[r]^-\times & K_1(S(A\otimes B)) \ar[d]^-{\alpha_{A\otimes B}} \\ K_0(A)\otimes K_0(B) \ar[r]^-{\times} & K_0(A\otimes B)}
$$
commutes; here the horizontal arrows are the canonical external products in $K$-theory discussed above, and we have used the canonical identifications 
$$
SA\otimes B=(S\otimes A)\otimes B=S\otimes (A\otimes B)=S(A\otimes B)
$$
to make sense of the top horizontal arrow.
\end{enumerate}
Then $\alpha_A$ and $\beta_A$ are mutually inverse for all $C^*$-algebras $A$. \qed
\end{proposition}

\section{Almost commuting matrices}

Our goal in this section is to construct homomorphisms $\alpha_A:K_1(SA)\to K_0(A)$ with the properties in Proposition \ref{atiyah rot} above, and thus prove Bott periodicity.   To motivate our construction, we start by recalling Voiculescu's almost commuting matrices.

Let $\{\delta_0,...,\delta_{n-1}\}$ be the canonical orthonormal basis for $\C^n$, and define unitaries by 
\begin{equation}\label{voic mat}
u_n:\delta_k\mapsto e^{2\pi i k/n} \delta_k \quad \text{and}\quad v_n:\delta_k\mapsto \delta_{k+1},
\end{equation}
where the `$k+1$' in the subscript above should be interpreted modulo $n$, or in other words $v_n:\delta_{n-1}\mapsto \delta_0$.  It is straightforward to check that $\|[u_n,v_n]\|\to 0$ as $n\to\infty$ (the norm here and throughout is the operator norm).  Voiculescu \cite{Voiculescu:1983km} proved the following result.

\begin{theorem}\label{voic the}
There exists $\epsilon>0$ such that if $u_n',v_n'$ are $n\times n$ unitary matrices that actually commute, then we must have $\max\{\|u_n-u_n'\|,\|v_n-v_n'\|\}\geq \epsilon$ for all $n$. \qed
\end{theorem}

In words, the sequence of pairs $(u_n,v_n)$ cannot be well-approximated by pairs that actually commute.  Voiculescu's original proof of this fact uses non-quasidiagonality of the unilateral shift; there is a close connection of this to $K$-theory and Bott periodicity, but this is not obvious from the original proof.  A more directly $K$-theoretic argument is due to Loring from his thesis \cite{Loring:1985ud}.  Fix smooth functions $f,g,h:S^1\to [0,1]$ with the properties in \cite[pages 10-11]{Loring:1985ud}: the most salient of these are that 
$$
f^2+g^2+h^2=f, \quad f(1)=1,\quad g(1)=h(1)=0, \quad \text{and}\quad gh=0.  
$$
For a pair of unitaries $u,v$ in a unital $C^*$-algebra $B$, define the \emph{Loring element} 
\begin{equation}\label{loring elt}
e(u,v):=\begin{pmatrix} f(u) & g(u)+ h(u)v \\ v^*h(u)+g(u) & 1-f(u)\end{pmatrix},
\end{equation}
which in general is a self-adjoint element of $M_2(B)$.  For later use, note that the conditions on $f,g,h$ above imply the following formula, where for brevity we write $e=e(u,v)$, $f=f(u)$, $g=g(u)$, $h=h(u)$,
\begin{equation}\label{e2 form}
e^2=e+\begin{pmatrix} hvg+gv^*h & [f,hv] \\ [v^*h,f] & v^*h^2v-h^2 \end{pmatrix};
\end{equation}
in particular, if $u$ and $v$ happen to commute, this shows that $e$ is an idempotent.  Similarly, one can show that if $\|uv-vu\|$ is suitably small, then $\|e^2-e\|$ is also small, so in particular the spectrum of $e$ misses $1/2$: indeed, this is done qualitatively in \cite[Proposition 3.5]{Loring:1985ud}, while a quantitative result for a specific choice of $f$, $g$, and $h$ can be found in \cite[Theorem 3.5]{Loring:2014xw}; the latter could be used to make the conditions on $t$ that are implicit in our results more explicit.  Thus if $\chi$ is the characteristic function of $[1/2,\infty)$, then $\chi$ is continuous on the spectrum of $e$, and so $\chi(e)$ is a well-defined projection in $B$.  Loring shows that if $e_n\in M_{2n}(\C)$ is the Loring element associated to the matrices $u_n,v_n\in M_n(\C)$ as in line \eqref{voic mat}, then for all suitably large $n$, 
$$
\text{rank}(e_n)-n=1.
$$
However, it is not difficult to see that if $u_n$ and $v_n$ were well-approximated by pairs of actually commuting matrices in the sense that the conclusion of Theorem \ref{voic the} fails, then for all suitably large $n$ one would have that $\text{rank}(e_n)=0$.  Thus Loring's work in particular reproves Theorem \ref{voic the}.

Now, let us get back to constructing a family of homomorphisms $\alpha_A:K_1(SA)\to K_0(A)$ satisfying the conditions of Proposition \ref{atiyah rot}, and thus to prove Bott periodicity.  For this, it will be convenient to have a continuously parametrised versions of the sequence $(u_n)$, which we now build.

Let then $\delta_n:S^1\to \C$ be defined by $\delta_n(x)=\frac{1}{\sqrt{2\pi}}z^n$, so the collection $\{\delta_n\}_{n\in \Z}$ is the canonical `Fourier orthonormal basis' for $L^2(S^1)\cong \ell^2(\Z)$.  For each $t\in [1,\infty)$, define a unitary operator $u_t:L^2(S^1)\to L^2(S^1)$ to be diagonal with respect to this basis, and given by
$$
u_t\delta_n:=\left\{\begin{array}{ll} e^{2\pi i nt^{-1}}\delta_n & 0\leq t^{-1}n \leq 1 \\  \delta_n & \text{otherwise} \end{array}\right..
$$
Thus $u_t$ agrees with the operator of rotation by $2\pi t^{-1}$ radians on $\text{span} \{\delta_0,...,\delta_n\mid n\leq t\}$, and with the identity elsewhere.  Let $A$ be a unital $C^*$-algebra faithfully represented on some Hilbert space $H$, and represent $C(S^1)$ on $L^2(S^1)$ by multiplication operators in the canonical way.  Represent $SA=\{f\in C(S^1, A)\mid f(1)=0\}$ faithfully on $L^2(S^1,H)$ as multiplication operators.  Let $\chi$ be the characteristic function of $\{z\in \C\mid \text{Re}(z)\geq \frac{1}{2}\}$.

Note that the Bott element $b$ acts on $L^2(S^1)$ via the (backwards) bilateral shift.  When compressed to the subspace $\text{span}\{\delta_0,...,\delta_n\mid n\leq t\}$ of $L^2(S^1)$, the operators $u_t$ and $b$ are thus slight variants of Voiculescu's almost commuting unitaries from line \eqref{voic mat}.

\begin{lemma}\label{in right place}
Let $A$ be a unital $C^*$-algebra, let $\widetilde{SA}$ be the unitization of its suspension, and let $v\in \widetilde{SA}$ be a unitary operator.  Then with notation as in the above discussion:
\begin{enumerate}[(i)]
\item the spectrum of the Loring element $e(u_t\otimes 1_A,v)$ (cf.\ line \eqref{loring elt}) does not contain $1/2$ for all large $t$;
\item the difference 
$$
\chi(e(u_t\otimes 1_A,v))-\begin{pmatrix} 1 & 0 \\ 0 & 0 \end{pmatrix}
$$
is in $\mathcal{K}(L^2(S^1))\otimes M_{2}(A)$ for all $t$;
\item the function 
$$
t\mapsto \chi(e(u_t\otimes 1_A,v))
$$
is operator norm continuous for all large $t$.
\end{enumerate}
\end{lemma}

\begin{proof}
For part (i), we first claim that for any $f\in C(S^1)$, $[v_t,f]\to 0$ as $t\to\infty$.  Indeed, it suffices to check this for the Bott element $b(x)=z^{-1}$ as this function generates $C(S^1)$ as a $C^*$-algebra.  With respect to the orthonormal basis $\{\delta_n:n\in \Z\}$ we have that $b$ acts by 
$$
b:\delta_n\mapsto \delta_{n-1},
$$
i.e.\ $b$ is the inverse of the usual bilateral shift.  On the other hand, we have that 
$$
\|[v_t,b]\|=\|(v_t-bv_tb^*)b\|=\|v_t-bv_tb^*\|,
$$
and one computes directly that $v_t-bv_tb^*$ is a multiplication operator by an element of $\ell^\infty(\Z)$ that tends to zero as $t$ tends to infinity, completing the proof of the claim.  

It follows from the claim that $[v_t\otimes 1,f\otimes a]\to 0$ as $t\to\infty$ for any $f\in C(S^1)$, and any $a\in A$.  Hence $[k(v_t)\otimes 1,a]\to 0$ for any $a\in C(S^1,A)$ and any $k\in C(S^1)$ (as $k(v_t)$ is in the $C^*$-algebra generated by $v_t$).  Part (i) follows from this and the formula in line \eqref{e2 form}, plus the fact that $hg=gh=0$.

For part (ii), consider
$$
e(u_t\otimes 1_A,v)=\begin{pmatrix} f(u_t)\otimes1_A & g(u_t)\otimes1_A+(h(u_t)\otimes 1_A)v \\ g(u_t)\otimes1_A+v^*(h(u_t)\otimes 1_A) & (1-f(u_t))\otimes1_A \end{pmatrix}.
$$
As $e$ is norm continuous in the `input unitaries', we may assume that $v\in C(S^1,A)$ is of the form $z\mapsto \sum_{n=-M}^M z^n a_n$ for some finite $M\in \N$.  It follows from this, the formula for $v_t$, and the facts that $h(1)=0=g(1)$ and $f(1)=1$, that there exists $N\in \N$ (depending on $M$ and $t$) such that the operator $e(u_t\otimes 1_A,v)$ leaves some subspace of $(L^2(S^1)\otimes H)^{\oplus 2}$ of the form 
$$
(\text{span}\{\delta_{-N},...,\delta_N\}\otimes H)^{\oplus 2}
$$
invariant, and moreover that it agrees with the operator $\begin{pmatrix} 1 & 0 \\ 0 & 0 \end{pmatrix}$ on the orthogonal complement of this subspace.  It follows from this that 
$$
\chi(e(u_t\otimes 1_A,v))-\begin{pmatrix} 1 & 0 \\ 0 & 0 \end{pmatrix}
$$
is also zero on the orthogonal complement of 
$$
(\text{span}\{\delta_{-N},...,\delta_N\}\otimes H)^{\oplus 2}
$$
and we are done.

For part (iii), it is straightforward to check that the functions $t\mapsto h(u_t)$, $t\mapsto f(u_t)$ and $t\mapsto g(u_t)$ are continuous, as over a compact interval in the $t$ variable, they only involve continuous changes on the span of finitely many of the eigenvectors $\{\delta_n\}_{n\in \Z}$.  It follows from this and the formula for $e(u,v)$ that the function 
$$
t\mapsto e(u_t\otimes 1_A,v)
$$
is norm continuous.  The claim follows from this, the fact that $\chi$ is continuous on the spectrum $e(u_t\otimes 1_A,v)$ for large $t$, and continuity of the functional calculus in the appropriate sense (see for example \cite[Lemma 1.2.5]{Rordam:2000mz}).
\end{proof}

\begin{corollary}\label{alpha map}
With notation as above, provisionally define 
$$
\alpha_A:K_1(SA)\to K_0(A\otimes \mathcal{K}) 
$$
by the formula for $u\in M_n(\widetilde{A})$
$$
[u]\mapsto [\chi(e(v_t\otimes 1_{M_n(A)},u))]-\begin{bmatrix} 1 & 0 \\ 0 & 0 \end{bmatrix},
$$
where $t$ is chosen sufficiently large (depending on $u$) so that all the conditions in Lemma \ref{in right place} hold.  Then this is a well-defined homomorphism for any $C^*$-algebra $A$.
\end{corollary}

\begin{proof}
This is straightforward to check from the formulas involved together with the universal property of $K_1$ as exposited in \cite[Proposition 8.1.5]{Rordam:2000mz}, for example.
\end{proof}

Abusing notation slightly, we identity $K_0(A\otimes \mathcal{K})$ with $K_0(A)$ via the canonical stabilization isomorphism, and thus treat $\alpha_A$ as a homomorphism from $K_1(SA)$ to $K_0(A)$.  To complete our proof of Bott periodicity, it remains to check that these homomorphisms $\alpha_A$ have the properties from Proposition \ref{atiyah rot}.  The second of these properties is almost immediate; we leave it to the reader to check.  

The first property, that $\alpha_\C(b)=1$, is more substantial, and we give the proof here following computations in Loring's thesis. 

\begin{proposition}\label{key computation}
With notation as above, $\alpha_\C(b)=1$.
\end{proposition}

\begin{proof}
We must compute the element
$$
[\chi(e(u_t,b))]-\begin{bmatrix} 1 & 0 \\ 0  & 0 \end{bmatrix} \in K_0(\mathcal{K})\cong \Z
$$
for suitably large $t$, and show that it is one.  We will work with an integer value $t=N$ for $N$ suitably large.  Note that with respect to the canonical basis $\{\delta_n\}_{n\in \Z}$ of $L^2(S^1)$, the element $b$ acts as the (inverse of the) unilateral shift.  On the other hand, on this basis $u_N(\delta_n)=\delta_n$ for all $n\not\in (0,N)$.  Define 
$$
H_N:=\text{span}\{\delta_n\mid 1\leq n \leq N\},
$$
It follows from the above observations, the fact that $f(1)=1$ and $h(1)=g(1)=0$, and a direct computation, that $H_N^{\oplus 2}$ is an invariant subspace of $L^2(S^1)^{\oplus 2}$ for both $e(u_N,b)$, and for $\begin{pmatrix} 1 & 0 \\ 0 & 0 \end{pmatrix}$.  Moreover, these two operators agree on the orthogonal complement $(H_N^{\oplus 2})^\perp$.  On $H_N^{\oplus 2}$, $e(u_N,b)$ agrees with the operator $e(u_N,b_N)$, where we abuse notation by writing $u_N$ also for the restriction of $u_N$ to $H_N$, and where $b_N:H_N\to H_N$ is the permutation operator defined by 
$$
b_N:\delta_n\mapsto \delta_{n+1~\text{mod $n$}}
$$
(i.e.\ $b_N$ is the cyclic shift of the canonical basis).  From these computations, we have that if we identify $K_0(\mathcal{K}(L^2(S^1)))\cong \Z$ and $K_0(\mathcal{B}(H_N))\cong \Z$ via the canonical inclusion $\mathcal{B}(H_N)\to \mathcal{K}(L^2(S^1))$ (which induces an isomorphism on $K$-theory) then 
$$
[\chi(e(u_N,b))]-\begin{bmatrix} 1 & 0 \\ 0  & 0 \end{bmatrix}= [\chi(e(u_N,b_N))]-\begin{bmatrix} 1 & 0 \\ 0  & 0 \end{bmatrix}\in \Z.
$$
Thus we have reduced the proposition (and therefore the proof of Bott periodicity) to a finite-dimensional matrix computation for $N$ large: we must show that if $e_N:=\chi(e(u_N,b_N))$, then the trace of the $2N\times 2N$ matrix
$$
\chi(e_N)-\begin{pmatrix} 1 & 0 \\ 0 & 0 \end{pmatrix}
$$
is $1$ for all large $N$, or equivalently, that the trace of the $2N\times 2N$ matrix $e_N$ is $N-1$ for all large $N$.

This can be computing directly, following Loring.  The computation is elementary, although slightly involved; we proceed as follows.\\

\textbf{Step 1:} $\|\chi(e_N)-e_N\|=O(1/N)$ for all $N$.

For notational convenience, we fix $N$, drop the subscript $_N$, and write $f$ for $f(v_N)=f(v)$ and similarly for $g$ and $h$.  We have that $|\chi(x)-x|\leq 2|x^2-x|$ for all $x\in \R$, whence from the functional calculus $\|\chi(e)-e\|\leq 2\|e-e^2\|$.   Using the formula in line \eqref{e2 form} above, it will thus suffice to show that the norm of 
$$
\begin{pmatrix} hbg+gb^*h & [f,bv] \\ [b^*h,f] & b^*h^2b-h^2 \end{pmatrix}
$$
is bounded by $C/N$, where $C>0$ is an absolute constant not depending on $N$.  Note that if a function $k:S^1\to \C$ is Lipschitz with Lipschitz constant $\text{Lip}(k)$ , we have that 
$$
\|[k,b]\|=\|bkb^*-k\|=\sup_{x\in [0,1]}|k(e^{2\pi ix})-k(2\pi i (x+1/N))|\leq \frac{2\pi\text{Lip}(k)}{N}.
$$
This, combined with the fact that $\|h\|\leq 1$ implies that 
$$
\Big\|\begin{pmatrix} h[b,g]+[g,b^*]h & h[f,b] \\ [u^*,f]h & u^*(h[h,b]+[h,b]h)\end{pmatrix}\Big\|\leq \frac{4\pi}{N}(\text{Lip}(f)+\text{Lip}(g)+\text{Lip}(h))
$$
and we are done with this step.\\

\textbf{Step 2:} $\text{tr}(\chi(e_N))-\text{tr}(3e_N^2-2e_N^3)\to 0$ as $N\to \infty$.

The result of step one says that there is a constant $C>0$ such that the eigenvalues of $e_N$ are all within $C/N$ of either one or zero for all large $N$.  The function $p(x)=3x^2-2x^3$ has the property that $p(0)=0$, $p(1)=1$, and $p'(0)=p'(1)=0$.  Hence there is a constant $D$ such that the eigenvalues of $p(e_N)$ are all within $D/N^2$ of either $0$ or $1$.  It follows that $\|\chi(e_N)-p(e_N)\|\leq D/N^2$.  Hence 
$$
|\text{tr}(\chi(e_N))-\text{tr}(p(e_N))|=|\text{tr}(\chi(e_N)-p(e_N))|\leq 2N\|\chi(e_N)-p(e_N)\|\leq \frac{2D}{N}.
$$
This tends to zero as $N\to\infty$, completing the proof of step 2.\\

\textbf{Step 3:} $\text{tr}(e_N^2)=N$ for all $N$.

Using the formula in line \eqref{e2 form} (with the same notational conventions used there) and rearranging a little, we get
$$
\text{tr}(e^2)=\text{tr}(e)+\text{tr}(hbg+gb^*h)+\text{tr}(b^*h^2b-h^2).
$$
From the formula for $e$, we get that $\text{tr}(e)=N$.  The second two terms are both zero using the trace property, and that $gh=0$.\\

\textbf{Step 4:} $\text{tr}(e_N^3)-(N-\frac{1}{2})\to 0$ as $N\to\infty$.

Again, we use the formula in line \eqref{e2 form}, multiplied by $e=e_N$ to see that 
$$
\text{tr}(e^3)=\text{tr}(e^2)+\text{tr}\Big(e\begin{pmatrix} hbg+gb^*h & [f,hb] \\ [b^*h,f] & b^*hb-h^2 \end{pmatrix}\Big).
$$
The first term is $N$ using step 3.  Multiplying the second term out, simplifying using the trace properties and that $gh=0$, we see that the trace of the second term equals 
$$
\text{tr}(3h^2(f-bfb^*))=3\sum_{k=0}^{N-1} h(e^{2\pi ik/N})^2(f(e^{2\pi i k/N})-f(e^{2\pi i(k-1)/N})).
$$
Assuming (as we may) that $f$ is differentiable, this converges as $N$ tends to infinity to 
$$
3\int_0^1 h(x)^2f'(x)dx.
$$ 
Using that $h=0$ on $[0,1/2]$, and that $h=f-f^2$ on $[1/2,1]$ (plus the precise form for $f$ in \cite[pages 10-11]{Loring:1985ud}) we get that 
$$
3\int_{\frac{1}{2}}^1 (f(e^{2\pi ix})-f(e^{2\pi ix})^2)f'(e^{2\pi ix})dx=3\int_0^1 \lambda-\lambda^2 d\lambda = \frac{1}{2}.
$$
This completes the proof of step 4.

Combining steps 2, 3, and 4 completes the argument and we are done.  
\end{proof}

\section{Connection with the localization algebra}

Homomorphisms $\alpha_A$ with the properties required by Proposition \ref{atiyah rot} are maybe more usually defined using differential, or Toeplitz, operators associated to the circle.  There is a direct connection between this picture and our unitaries $u_t$, which we now explain; we will not give complete proofs here as this would substantially increase the length of the paper, but at least explain the key ideas and connections.  

The following definition is inspired by work of Yu \cite{Yu:1997kb} in the case that $A$ is commutative.  It agrees with Yu's original definition for unital commutative $C^*$-algebras.

\begin{definition}\label{loc alg}
Let $A$ be a $C^*$-algebra, and assume that $A$ is represented nondegenerately and essentially\footnote{This means that no non-zero element of $A$ acts as a compact operator, so in particular, such a representation is faithful.} on some Hilbert space $H$.  Let $C_{ub}([1,\infty),\mathcal{B}(H))$ denote the $C^*$-algebra of bounded uniformly continuous functions from $[1,\infty)$ to $\mathcal{B}(H)$.  The \emph{localization algebra} of $A$ is defined to be
$$
C_L^*(A):=\left\{ \begin{array}{l|l} f\in C_{ub}([1,\infty),\mathcal{B}(H)) &  f(t)a\in \mathcal{K}(H) \text{ for all }t\in [1,\infty),a\in A \\ &\text{ and } \|[f(t),a]\|\to 0 \text{ for all } a\in A\end{array}\right\}.
$$
If $A=C_0(X)$ is commutative, we will write $C_L^*(X)$ instead of $C^*_L(C_0(X))$.
\end{definition}

The localization algebra of a $C^*$-algebra $A$ does not depend on the choice of essential representation $H$ up to non-canonical isomorphism, and its $K$-theory does not depend on $H$ up to canonical isomorphism (these remarks follow from \cite[Theorem 2.7]{Dadarlat:2016qc}); thus we say `the' localization algebra of $A$, even thought this involves a slight abuse of terminology.  Moreover, building on work of Yu \cite{Yu:1997kb} and Qiao-Roe \cite{Qiao:2010fk} in the commutative case, \cite[Theorem 4.5]{Dadarlat:2016qc} gives a canonical isomorphism 
$$
K^*(A)\to K_*(C^*_L(A))
$$
from the $K$-homology groups of $A$ to the $K$-theory groups of $C^*_L(A)$ (at least when $A$ is separable).  

One can define a pairing 
$$
K_i(C_L^*(A))\otimes K_j(A)\mapsto \Z
$$
between the $K$-theory of the localization algebra (i.e.\ $K$-homology) and the $K$-theory of a $C^*$-algebra $A$.  The most complicated case, and also the one relevant to the current discussion  of this occurs when $i=j=1$, so let us focus on this.  Let $(u_t)_{t\in [1,\infty)}$ be a unitary in the unitization of $C^*_L(A)$ and $v$ in the unitization of $A$ be another unitary (the construction also works with matrix algebras involved in exactly the same way, but we ignore this for notational simplicity).  Let $H$ be the Hilbert space used in the definition of the localization algebra, and for $t\in [1,\infty)$, let  
$$
e(u_t,v)\in M_2(\mathcal{B}(H))
$$
be the Loring element of line \eqref{loring elt}.  One can check that for all large $t$, the spectrum of $e(u_t,v)$ does not contain $1/2$.  Hence if $\chi$ be the characteristic function of $[1/2,\infty)$, we get a difference
$$
\chi(e(u_t,v))-\begin{pmatrix} 1 & 0 \\ 0 & 0 \end{pmatrix}\in M_2(\mathcal{B}(H))
$$
of projections for all suitably large $t$.  It is moreover not difficult to check that this difference is in $M_2(\mathcal{K}(H))$, and thus defines an element in $K_0(\mathcal{K}(H))\cong \Z$, which does not depend on $t$ for $t$ suitably large.  We may thus define 
$$
\langle [u_t],[v]\rangle:=[e(u_t,v)]-\begin{bmatrix} 1 & 0 \\ 0 & 0 \end{bmatrix} \in K_0(\mathcal{K})\cong \Z
$$
for any suitably large $t$.  One checks that this formula gives a well-defined pairing between $K_1(A)$ and $K^1(A)$.  More substantially, one can check that it agrees with the canonical pairing between $K$-homology and $K$-theory (at least up to sign conventions).  

Let us go back to the case of interest for Bott periodicity.  In terms of elliptic operators, one standard model for the canonical generator of the $K$-homology group $K_1(S^1)$ of the circle is the class of the Dirac operator $D=\frac{-i}{2\pi}\frac{d}{d\theta}$, where $\theta$ is `the' angular coordinate.  We consider $D$ as an unbounded operator on $L^2(S^1)$ with domain the smooth functions $C^\infty(S^1)$.  Let $\chi:\R\to [-1,1]$ be any continuous function such that $\lim_{t\to\pm\infty}\chi(t)=\pm 1$, and for $t\in [1,\infty)$, define
$$
F_t:=\chi(t^{-1}D)
$$
using the unbounded functional calculus.  Concretely, each $F_t$ is diagonal with respect to the canonical orthonormal basis $\{\delta_n\mid n\in \Z\}$, acting by 
$$
F_t:\delta_n\mapsto \chi(t^{-1}n)\delta_n~;
$$
this follows as $D:\delta_n\mapsto n\delta_n$ for each $n$.  

Using the above concrete description of the eigenspace decomposition of $F_t$, it is not too difficult to show that the function $t\mapsto F_t$ defines an element of the multiplier algebra $M(C^*_L(S^1))$ of the localization algebra $C_L^*(S^1)$.  Moreover, one checks similarly that the function $t\mapsto \frac{1}{2}(F_t+1)$ maps to a projection in $M(C_L^*(S^1)) / C_L^*(S^1)$, and thus defines a class $[D]_0\in K_0(M(C_L^*(S^1))/C_L^*(S^1))$.  Yu defines the $K$-homology class associated to this operator to be the image $[D]$ of $[D]_0$ under the boundary map
$$
\partial : K_0(M(C_L^*(S^1))/C_L^*(S^1)) \to K_1(C_L^*(S^1))
$$
(all this is part of a very general machine for turning elliptic differential operators into elements of $K_*(C^*_L(S^1))$: see \cite[Chapter 8]{Willett:2010ay}).  This boundary map is explicitly computable (compare for example \cite[Section 12.2]{Rordam:2000mz}): the image of $[D]_0$ under this map is the class of the unitary
$$
e^{2\pi i\frac{1}{2}(F_t+1)}=-e^{\pi i \chi(t^{-1}D)}.
$$
Choosing $\chi$ to be the function which is negative one on $(-\infty,0]$, one on $[1,\infty)$, and that satisfies $\chi(t)=2t-1$ on $(0,1)$, we see that 
\begin{equation}\label{ut form}
-e^{\pi i \chi(t^{-1}D)}=u_t
\end{equation}
i.e.\ the canonical generator of the $K$-homology group $K_1(S^1)$ is given precisely by the class of $u_t$ in $K_1(C_L^*(S^1))$.  

The formula for $\alpha_A$ is then a specialization of the general formula above for the pairing, using the element of line \eqref{ut form} (more precisely, its amplification to $C(S^1,A)$ in an obvious sense).  Thus our proof of Bott periodicity using almost commuting matrices fits very naturally into the localization picture of $K$-homology.

\bibliography{Generalbib}

\end{document}